\documentclass{amsart}
\usepackage[english]{babel}

\usepackage{amsfonts}
\usepackage{amssymb}
\usepackage{color}

\newtheorem{thm}{Theorem}[section]
\newtheorem{cor}[thm]{Corollary}
\newtheorem{lem}[thm]{Lemma}

\theoremstyle{definition}

\newtheorem{exmp}{Example}[section]

\theoremstyle{remark}
\newtheorem{rem}{Remark}[section]

\numberwithin{equation}{section}

\begin{document}

\title[On approximation of solutions]{On approximation of solutions to the heat 
equation from Lebesgue class $L^2$ by more regular solutions}

\author{A.A. Shlapunov}

\address[Alexander Shlapunov]
{Siberian federal university
                                                 \\
         pr. Svobodnyi, 79
                                                 \\
         660041 Krasnoyarsk
                                                 \\
         Russia}
\email{ashlapunov@sfu-kras.ru}

\address{Sirius Mathematics Center \\
Sirius University of Science and Technology
                                                 \\
          Olimpiyskiy ave. b.1
                                                 \\
         354349 Sochi
                                                 \\
         Russia}
\email{shlapunov.aa@talantiuspeh.ru}

\begin{abstract}
Let $s \in {\mathbb N}$,   
$T_1,T_2 \in {\mathbb R}$, $T_1<T_2$, and $\Omega, \omega $ be bounded domains
in ${\mathbb R}^n$, $n \geq 1$, such that  $\omega \subset \Omega$ and 
the complement $\Omega \setminus \omega$ has no (non-empty) 
compact  components in  $\Omega$. We prove that this is the necessary and sufficient 
condition for the  space $H^{2s,s}  _{\mathcal H} (\Omega \times (T_1,T_2))$ of 
solutions to the heat operator  ${\mathcal H} $ in a cylinder domain 
$\Omega \times (T_1,T_2)$ from the anisotropic Sobolev  space   
$H^{2s,s}  (\Omega \times (T_1,T_2))$  to be dense in the space 
$L^{2} _{\mathcal H}(\omega \times (T_1,T_2))$,  consisting of solutions in the domain 
$\omega \times (T_1,T_2)$ from the Lebesgue class  $L^{2}  (\omega \times (T_1,T_2))$. 
As an important corollary we obtain the theorem on the existence of a basis 
with the double orthogonality property for the pair of the Hilbert spaces 
$H^{2s,s}  _{\mathcal H} (\Omega \times (T_1,T_2))$ and $L^{2} 
_{\mathcal H}(\omega \times (T_1,T_2))$ .
\end{abstract}

\maketitle

\section*{Introduction}
\label{s.Int}

The problem of the uniform approximation on subcompacts of a domain in ${\mathbb R}^{n+1}$
of solutions to the heat equation was solved in the 
papers by \cite{J}, \cite{D80} (see also \cite{GauTa10} for 
some refinement related to the so-called rational approximation). 
It appeared that for this purpose one may use an approach quite similar 
to   the  Runge type approximations of solutions to an elliptic 
system in a lesser domain by solutions in a bigger domain 
 (these include the theorem  by C. Runge  \cite{R1885} for the holomorphic 
functions, the theorem by S.N. Mergelyan  \cite{Merg56} for the harmonic functions
and their  generalizations for  spaces of solutions 
to  various systems of partial differential equations (see, for instance, 
\cite{Mal56} for operators with constant coefficients
or \cite[ch. 4, ch. 5]{Tark37} for elliptic operators with sufficiently smooth 
coefficients). More precisely, the key assumption, providing 
that the space  $S _{\mathcal H}(D)$ of solutions to the heat operator  
${\mathcal H}$  to be dense in the space  
$S _{\mathcal H}(D')$ for a pair of domains $D'\subset D \subset {\mathbb R}^{n+1}$, 
is the absence of compact  components of  $D_t \setminus D'_t$ in $D_t$ for any sections 
$D_t, D'_t$ of the domains $D$ and $D'$, respectively, by hyperplanes parallel 
to  the coordinate hyperplane $\{ \tau=0 \}$ in ${\mathbb R}^{n+1}$ 
and  containing the point $t$. 

However, more delicate approximation problems  
appeared in the middle of the last century in the theory of the analytic functions.
They were related to the approximation in the function spaces where the behaviour 
of the elements are controlled up to the boundaries of the considered sets, see, 
for instance, pioneer papers by  A.G. Vitushkin \cite{V67} and V.P. Havin \cite{H68}.  
Later on it was discovered that these problems of approximation by 
solutions of various differential equations are closely 
related to the theory of non-linear potential, see, for instance,
\cite{Hedb86}.

In the present paper we will discuss approximation of solutions to the 
heat equation from the normed Lebesgue space  $L^{2} _{\mathcal H}(D)$. 
This imposes certain restrictions on the domain 
 $D \subset {\mathbb R}^{n+1}$. 
To simplify the exposition we will concentrate our efforts on 
cylinder domains with the forming side surfaces  parallel to the time axis.

One of the reason to consider this problem is the need to construct bases 
with the double orthogonality property related the spaces of solutions 
to the heat equations in a pair cylinder domains. We recall that systems with 
the double orthogonality property are used to investigate operator equations 
in Hilbert spaces since the middle of the XX-th century,  see, for instance, 
\cite{Kras68}. They were especially useful in the situations where 
the linear operator $A: H_1 \to H_2$, acting between Hilbert spaces $ H_1, H_2$, 
is injective, compact and it has a dense image. In this case, the basis with the double 
orthogonality property with respect to two inner products   
$(\cdot, \cdot)_{H_1}$ and $(A\cdot, A\cdot)_{H_2}$ 
on the space $ H_1$, is the complete system of eigen-vectors of the compact self-adjoint 
operator $A^* A: H_1 \to H_1$. This allows to construct 
regularising operators  for the operator equation $Au=f$ with given 
$f\in H_2$ and looked for  $u \in H_1$, 
see, for instance,   \cite{TikhArsX}, 
\cite[ch. 12]{Tark36}. 

At the first part of the  XX-th century, long before the formation of the this standard 
scheme of the functional analysis,  S. Bergman suggested to use such systems for spaces 
of holomorphic functions in the problem of the analytic continuation from 
a lesser plane domain to a bigger one (see later exposition by \cite{Berg70}). 
Later, in  1980', L. Aizenberg, using integral representation method, reduced the Cauchy 
problem for holomorphic functions (of one and several variables) to the problem 
of the analytic continuation. This opened the way for applying the systems with the double 
orthogonality for the construction of the Carleman formulas, see \cite{AKy}. 
This approach was successfully use in order to investigate the Cauchy problem for a wide 
class of elliptic (including the overdetermined ones) systems with real analytic coefficients,
 see  \cite[ch. 10, ch. 12]{Tark36}, 
\cite{ShTaLMS}, 
and even to elliptic differential complexes, see  \cite{FeSh2}). Taking in the account 
the  requirements described above for a continuous operator
$A: H_1 \to H_2$, the key issues for this theory were the Uniqueness Theorems for 
solutions to elliptic systems  with real analytic coefficients, 
providing its injectivity, the Sobolev Embedding Theorems, Rellich-Kondrashov 
Theorem and/or Stiltjes-Vitaly Theorem, guaranteeing
its compactness, and the Runge type Theorems on the approximation of solutions 
in a lesser domain by the solutions in a bigger one. 

It appeared that the Cauchy problem for the heat operator  
${\mathcal H} = \frac{\partial}{\partial t} - 
\sum_{j=1}^n \frac{\partial^2}{\partial x^2_j}$ in a cylinder domain 
 ${\mathbb R}^n \times \mathbb R$ with the Cauchy data on a part of its 
lateral surface (that naturally arises in the diffusion problems, for instance 
in the inverse problem of the electrocardiography using models of the charge diffusion 
in the heart tissues) can be also reduced to the continuation problem for solutions 
of the heat equations from a lesser cylinder domain to a bigger one, see 
\cite{MMT17}, \cite{KuSh}. Since the solutions to the heat equation are real analytic 
with respect to the space variables  (see, for instance,  
\cite[ch. VI, \S 1, theorem 1]{MikhX}), then for domains 
$\omega \subset \Omega \subset {\mathbb R}^n$ and numbers
 $T_1,T_2 \in {\mathbb R}$, $T_1<T_2$, the natural embedding operator 
$$
A: H^{2s,s}  _{\mathcal H} (\Omega \times (T_1,T_2)) \to 
H^{2s',s'} _{\mathcal H}(\omega \times (T_1,T_2)), \,\, s\geq s', s,s'\in {\mathbb Z}_+, 
$$ 
between anisotropic spaces 
 $H^{2s,s}  _{\mathcal H} (\Omega \times (T_1,T_2))$ and  
$H^{2s',s'} _{\mathcal H}(\omega \times (T_1,T_2))$,  consisting of solutions to the 
heat operator  ${\mathcal H}$ from the  Sobolev class 
$H^{2s,s}  (\Omega \times (T_1,T_2))$ in a cylinder domain 
$\Omega \times (T_1,T_2) \subset {\mathbb R}^{n+1}$ 
and the Sobolev $H^{2s',s'}  (\omega \times (T_1,T_2))$ in the domain 
$\omega \times (T_1,T_2)$, respectively, is injective. The compactness of the 
operator $A$ may be easily extracted from general embedding theorems 
for anisotropic Sobolev type spaces, see, for instance, 
\cite[ch. III and ch. VI]{BIN}, or from the results by 
J.-L. Lions \cite[ch. 1, \S 5]{Lion69}, see \S \ref{s.ex} below. 
As for the density of the range of the operator $A$, it is precisely connected with the 
approximation theorems for the solutions to the heat equation.

\section{A density theorem}
\label{s.2}

Let ${\mathbb R}^n$ be $n$-dimensional Euclidean space with the coordinates 
$x=(x_1, \dots , x_n)$, and  $\Omega \subset {\mathbb R}^n$ be a bounded domain.
As usual, denote by  $\overline{\Omega}$ the closure of $\Omega$, and by
 $\partial\Omega$ its boundary. We assume that  $\partial \Omega$ 
is piece-wise smooth hypersurface. We denote also by 
$\Omega_{T_1,T_2}$ a bounded open cylinder  
 $\Omega \times (T_1,T_2)$ in ${\mathbb R}^{n+1}$ with $T_1<T_2$. 
 
We consider functions over  ${\mathbb R}^n$ and 
${\mathbb R}^{n+1}$.  For $s \in {\mathbb Z}_+$, we denote by $C^s(\Omega)$ 
the space of all $s$-times continuously differentiable functions  
over $\Omega$, and by $C^s(\overline \Omega)$ the subset of   
$C^s(\Omega)$ such that for each function $u \in C^s(\overline \Omega)$ and each 
multi-index  $\alpha = (\alpha_1, \dots , \alpha_n )
\in {\mathbb Z}^n_+$ there is a function $u_\alpha$, 
continuous on  $\overline \Omega$ and satisfying $\partial ^\alpha u = u_\alpha$ 
in  $\Omega$. Besides, let $L^2 (\Omega)$ be the Lebesgue space 
over $\Omega$ with the standard inner product  
$
(u,v)_{L^2 (\Omega)} 
$, 
and $H^s (\Omega)$ be the Sobolev space, $s\in \mathbb N$, 
with the standard inner product  
$
(u,v)_{H^s (\Omega)}
$.  
As it is well known, both $L^2 (\Omega)$ and $H^s (\Omega)$ are  Hilbert spaces.

Investigating spaces of solutions to the heat equation, we need 
the anisotropic Sobolev spaces $H^{2s,s} (\Omega_{T_1,T_2})$, $s \in  {\mathbb Z}_+$, see, 
for instance, 
\cite[ch. 2]{Kry08}, i.e. the set of all measurable functions 
 $u$ over $\Omega_{T_1,T_2}$ such that (generalised) partial derivatives  
$\partial^j_t \partial^{\alpha}_x u$ belong to the Lebesgue space $L^{2} (\Omega_{T_1,T_2})$
for all multi-indexes $(\alpha,j) 
\in {\mathbb Z}_+^{n} \times {\mathbb Z}_+$ with $|\alpha|+2j \leq 2s$. 
This is a Hilbert space with the inner product 
\begin{equation*} 
(u,v)_{H^{2s,s} (\Omega_{T_1,T_2})} = \sum_{|\alpha|+2j \leq 2s}
\int_{\Omega_{T_1,T_2}} \partial^j_t \partial^{\alpha}_x v 
(x,t) \, \partial^j_t \partial^{\alpha}_x u (x,t) dx dt.
\end{equation*}
For $s=0$ we obtain  $H^{0,0} (\Omega_{T_1,T_2}) = L^{2} (\Omega_{T_1,T_2})$. 
In particular,  
$C^{\infty} (\overline{\Omega_{T_1,T_2}})$ can be considered as the intersection 
 $\cap _{s=0}^\infty H^{2s,s} (\Omega_{T_1,T_2})$. 

Finally, for $k \in {\mathbb Z}_+$, we denote by $H^{k,2s,s} (\Omega_{T_1,T_2})$ 
the set of all functions  $u \in H^{2s,s} (\Omega_{T_1,T_2})$ such that 
 $\partial ^\beta_x u \in H^{2s,s} (\Omega_{T_1,T_2})$ for all  $|\beta|\leq k$. 
This is a Hilbert space with the inner product 
\begin{equation*} 
(u,v)_{H^{k,2s,s} (\Omega_{T_1,T_2})} = \sum_{|\beta| \leq k}
(\partial ^\beta u, \partial ^\beta v)_{H^{2s,s} (\Omega_{T_1,T_2})}.
\end{equation*}
We also will use the so-called Bocher spaces 
of functions depending on $(x,t)$ from the strip  
$\mathbb{R}^n \times  [T_1,T_2]$.
Namely, for a Banach space $\mathcal B$ (for example, the space of functions 
on a subdomain of $\mathbb{R}^n$) and    $p \geq 1$, we denote by  
$L^p (I,{\mathcal B})$ the Banach space of all the measurable mappings 
  $u : [T_1,T_2] \to {\mathcal B}$
with the finite norm  
$$
   \| u \|_{L^p ([T_1,T_2],{\mathcal B})}
 := \| \|  u (\cdot,t) \|_{\mathcal B} \|_{L^p ([T_1,T_2])},
$$
see, for instance, \cite[ch. \S 1.2]{Lion69},  \cite[ch.~III, \S~1]{Tema79}. 

The space $C ([T_1,T_2],{\mathcal B})$ is introduced with the use of the 
same scheme; this is the Banach space of all the measurable mappings
$u : [T_1,T_2] \to {\mathcal B}$ with the finite norm 
$$
   \| u \|_{C ([T_1,T_2],{\mathcal B})}
 := \sup_{t \in [T_1,T_2]} \| u (\cdot,t) \|_{\mathcal B}.
$$

Now let $S _{\mathcal H}(\Omega_{T_1,T_2})$ be the set of all generalised 
functions over $\Omega_{T_1,T_2}$, satisfying 
the (homogeneous) heat equation 
\begin{equation} \label{eq.heat}
{\mathcal H} u = 0 \mbox{ in } \Omega_{T_1,T_2} 
\end{equation}
in the sense of distributions. First of all, we note that according 
to the hypoellipticity  of the operator $\mathcal H$, solutions to equation 
\eqref{eq.heat} are infinitely differentiable 
on their domains (see, for instance,  \cite[ch. VI, \S 1, theorem 1]{MikhX}), i.e. 
$$
S _{\mathcal H}(\Omega_{T_1,T_2}) \subset C^\infty (\Omega_{T_1,T_2}).
$$
As it is known, this is a closed subspace in the space 
$C^\infty (\omega_{T_1,T_2})$ with the standard Fr\'echet topology (inducing the uniform 
convergence together with all the partial derivatives on compact subsets 
of $\omega_{T_1,T_2}$). 

Also, we need the space  $S _{\mathcal H}(\overline{\Omega_{T_1,T_2}})$, 
defined as the union of the spaces 
$$
\cup_{\tilde \Omega_{\tilde T_1,\tilde T_2} \supset 
\overline{\Omega_{T_1,T_2}} } S _{\mathcal H}(\tilde \Omega_{\tilde T_1,\tilde T_2}),
$$
where the union is with respect to all the domains 
$\tilde \Omega_{\tilde T_1,\tilde T_2}$, containing the closure of the domain 
 $\Omega_{T_1,T_2}$. Usually, this space is endowed with the topology 
of the inductive limit associated with a decreasing sequences of neighbourhoods of 
the compact  $\overline{\Omega_{T_1,T_2}}$; however, we will not 
use any topology of this space, considering it as a set of functions. 

Let $H^{k,2s,s} _{\mathcal H}(\Omega_{T_1,T_2}) =
H^{k,2s,s} (\Omega_{T_1,T_2}) \cap S _{\mathcal H}(\Omega_{T_1,T_2}) $, 
$s \in  {\mathbb Z}_+$, $k \in {\mathbb Z}_+$. As it 
is known, this is a closed subspace of the Sobolev space $H^{k,2s,s}(\Omega_{T_1,T_2})$. 
Similarly, $C^{\infty} _{\mathcal H} (\overline{\Omega_{T_1,T_2}}) =
C^{\infty} (\overline{\Omega_{T_1,T_2}}) \cap S _{\mathcal H}(\Omega_{T_1,T_2}) $ 
is a closed subspace, consisting of solutions to equation \eqref{eq.heat}, of the space 
$C^{\infty}  (\overline{\Omega_{T_1,T_2}})$ with the standard Fre\'echet topology. 
The hypoellipticity  of the operator $\mathcal H$ provides 
the following (continuous) embeddings 
\begin{equation} \label{eq.emb.H}
S _{\mathcal H}(\overline{\Omega_{T_1,T_2}}) \subset 
C^{\infty} _{\mathcal H} (\overline{\Omega_{T_1,T_2}}) \subset
H^{k,2s,s} _{\mathcal H}(\Omega_{T_1,T_2}) 
\end{equation} 
for all $k,s\in {\mathbb Z}_+$. 

Now we may formulate the main results of this paper.

\begin{thm} \label{t.dense.base}
If $\omega \subset \Omega \Subset {\mathbb R}^n$, $\partial \omega, \partial \Omega 
\in C^2$ then $S_{\mathcal H}(\overline {\Omega_{T_1,T_2}})$
is everywhere dense in $L^{2} _{\mathcal H}(\omega_{T_1,T_2})$ 
if and only if the complement $\Omega \setminus  \omega$ has no compact components in 
$\Omega$. 
\end{thm}

\begin{proof} \textit{Sufficiency.} Clearly, the set 
$S_{\mathcal H}(\overline {\Omega_{T_1,T_2}})$ is everywhere dense in 
 $L^{2} _{\mathcal H}(\omega_{T_1,T_2})$ if and only if the following 
relation
\begin{equation} \label{eq.ort}
(u,w)_{L^2 (\omega_{T_1,T_2})} = 0 \mbox{ for all } w\in 
S_{\mathcal H}(\overline {\Omega_{T_1,T_2}})
\end{equation}
means precisely for a function $u\in L^{2} _{\mathcal H}(\omega_{T_1,T_2})$ that 
$u=0$ in $\omega_{T_1,T_2}$. Of course, the zero function of the space 
$ L^{2} _{\mathcal H}(\omega_{T_1,T_2})$ satisfies \eqref{eq.ort}. 

Assume that the complement 
$\Omega \setminus  \omega$ has no (non-empty connected) compact components in  
$\Omega$. In order to prove the sufficiency of the statement 
we will use the fact that the heat operator ${\mathcal H}$  admits the bilateral 
fundamental solution of the convolution type, see, for instance,  \cite{MikhX,frid}:
\begin{equation*}
\Phi(x,t)=\begin{cases}
\frac{e^{-\frac{|x|^2}{4  \, t}} }{\left(2\sqrt{\pi \, t }\right)^n } & 
\mbox{ if } t>0,\\ 0 & \mbox{ if } t\leqslant   0.
\end{cases} 
\end{equation*} 
By the definition, 
\begin{equation} \label{eq.right}
{\mathcal H}_{x,t} \Phi(x-y,t-\tau) = \delta (x-y, t-\tau), \,\, 
\end{equation}
\begin{equation} \label{eq.left}
{\mathcal H}'_{y,\tau} \Phi(x-y,t-\tau)  =\delta (x-y, t-\tau), 
\end{equation}
where ${\mathcal H}'_{y,\tau}  = -\frac{\partial}{\partial \tau} - 
\sum_{j=1}^n \frac{\partial^2}{\partial y^2_j}$ is the formal adjoint operator  for 
$\mathcal H$ and 
$\delta (x, t)$ is the Dirac functional with the support at the point $(x,t)$.

Let for a function $u\in L^{2} _{\mathcal H}(\omega_{T_1,T_2})$ relation
\eqref{eq.ort} holds true. 

Consider an auxiliary function
\begin{equation}
\label{eq.v}
v (y,\tau) = \int_{\omega_{T_1,T_2}} u (x,t)  \Phi(x-y,t-\tau) dx \, dt.  
\end{equation}
According to \eqref{eq.right}, 
$$
{\mathcal H}_{x,t} \Phi(x-y,t-\tau) =0 \mbox{ if } (x,t) \ne (y,\tau).
$$
Hence 
$$
{\mathcal H}_{x,t} \Phi(x-y,t-\tau) =0 \mbox{ in } \Omega_{T_1,T_2}
$$
for any fixed pair $(y,\tau) \not \in \Omega_{T_1,T_2}$. Then, using 
the hypoellipticity of the operator $\mathcal H$, we conclude that 
$\Phi(x-y,t-\tau) \in S_{\mathcal H}(\overline {\Omega_{T_1,T_2}})$ 
with respect to $(x,t) \in \Omega_{T_1,T_2}$ for any fixed pair $(y,\tau) \not \in 
\Omega_{T_1,T_2}$. In particular, relation 
 \eqref{eq.ort}  implies 
\begin{equation} \label{eq.v1}
v (y,\tau) = 0   \mbox{ in }{\mathbb R}^{n+1} \setminus \overline \Omega_{T_1,T_2}.
\end{equation}
On the other hand, by property \eqref{eq.left} of the fundamental solution, 
\begin{equation} \label{eq.v0}
{\mathcal H}'v    = \chi_{\omega_{T_1,T_2}} u \mbox{ in } {\mathbb R}^{n+1},
\end{equation}
where $\chi_{\omega_{T_1,T_2}}$ is the characteristic function 
of the domain  $\omega_{T_1,T_2}$. Obviously, 
$$
{\mathcal H}'_{y,\tau} v (y,\tau)=0 \mbox{ in  } D \times (t_1,t_2) 
$$
for some domain  $D\subset {\mathbb R}^n$ and numbers  $t_1<t_2$
if and only if 
$$
{\mathcal H}_{y,\tau} v (y,-\tau)=0 \mbox{ in  } D \times (-t_2,-t_1) .
$$
Therefore, according to  \cite[ch. VI, \S 1, theorem 1]{MikhX}), 
the function $v$ is real analytic with respect to the space variables 
of ${\mathbb R}^{n+1} \setminus \overline \omega_{T_1,T_2} $. 

Since the domain $\omega$ has a smooth boundary, each component of 
${\mathbb R}^{n} \setminus \overline \omega $ is itself a non-empty open domain 
with a smooth boundary and the similar fact is true for the domain  $\Omega $. 
However the complements
$\Omega \setminus  \omega$ has no compact components  in $\Omega$ and hence
each component of ${\mathbb R}^{n} \setminus \omega $ 
intersects with ${\mathbb R}^{n} \setminus  \Omega $ by a non-empty open set.
Thus,  \eqref{eq.v1} and the uniqueness theorem for real analytic functions yields 
\begin{equation} \label{eq.v2}
v (y,\tau) = 0   \mbox{ in }{\mathbb R}^{n+1} \setminus \overline \omega_{T_1,T_2}.
\end{equation}
Besides,  \eqref{eq.v0}, \eqref{eq.v2} mean that the function  
$\tilde v (y,\tau) = v (y,-\tau) $ is a solution to the Cauchy problem 
$$
\left\{
\begin{array}{lll}
{\mathcal H} \tilde v = \chi_{\omega_{T_1,T_2}} u \mbox{ in } {\mathbb R}^{n}  \times (-T_2-1,1-T_1),\\
\tilde v(y,-T_2-1) = 0 \mbox{ on } {\mathbb R}^{n} .\\
\end{array}
\right.
$$
Then, according to \cite[ch. 2, \S 5, theorem 3]{Kry08}, 
$\tilde v\in H^{2,1} ({\mathbb R}^{n}  \times (-T_2-1,1-T_1))$, and the solution 
in this class is unique. Moreover, the regularity of this unique solution to the Cauchy 
problem may be expresses in terms of the Bochner spaces, too. Namely, 
 $\tilde v\in C ( [-T_2-1,1-T_1], H^{1}({\mathbb R}^{n}) )
\cap L^2 ([-T_2-1,1-T_1], H^{2} ({\mathbb R}^{n}) )
$, see, for instance, \cite[ch. 3, \S 1]{Tema79}, where  similar linear problems 
for Stokes equations are considered. In particular, 
the function $v$ belongs to the space 
\begin{equation} \label{eq.space.1}
 C ( [T_1-1,T_2+1], H^{1}({\mathbb R}^{n}) )
\cap L^2 ([T_1-1, T_2+1], H^{2} ({\mathbb R}^{n}) )
\cap H^{2,1} ({\mathbb R}^{n}  \times (T_1-1,T_2+1))  .
\end{equation}

\begin{lem} \label{l.dense.c0} Any function of the type  \eqref{eq.v}, satisfying 
\eqref{eq.v2}, may be approximated by elements of $C^\infty_0  (\omega_{T_2,T_2})$
in the topology of the Hilbert space $H^{2,1} (\omega_{T_1,T_2})$. 
\end{lem}

\begin{proof}
First of all, we note that such a function 
may be approximated by functions of the class 
$C^\infty_0  ({\mathbb R}^{n+1})$ in the topology of the space 
$H^{2,1} ({\mathbb R}^{n+1})$. This fact can be extracted from 
\cite[ch. 3, \S 7, property 6]{MikhX}, but it can be proved directly, too. 

Indeed, denote by $h_\delta (x)$ the standard compactly supported function with 
the support in the  ball  $B(x,\delta) \subset {\mathbb R}^n$ with the centre  
at the point  $x$ and of the radius   $\delta>0$:
$$
h_\delta (x) = \left\{
\begin{array}{lll}
0, & \mathrm{if} & |x|\geq \delta, \\
c(\delta) \exp{(1/(|x|^2-\delta^2))}, & \mathrm{if} & |x|<\delta, \\
\end{array}
\right.
$$
where $c(\delta) $ is  the constant providing equality 
$$
\int_{{\mathbb R}^n} h(x) \, dx =1.
$$
Then, as it is well known, the standard regularisation  
$$
(R _\delta v) (x,t) = \int_{{\mathbb R}^{n+1}} h_\delta (x-y,t-\tau) v (y,\tau) \, dy d\tau 
$$
belongs to the space  $C^\infty_0  ({\mathbb R}^{n+1})$ for any positive number 
$\delta$ and 
$$
\lim_{\delta \to + 0}\| v -R _\delta v  \|_{L^2 ({\mathbb R}^{n+1})} =0,
$$
see, for instance, \cite{MikhX}. Since the standard regularization is defined 
with the use of the convolution, and the function  $v$ belongs to 
$H^{2,1} ({\mathbb R}^{n+1})$ and supported in 
$\overline{\omega_{T_1,T_2}}$, then 
$$
\partial_x^\alpha \partial_t^j (R _\delta v) (x,t) = 
 (R _\delta \partial_y^\alpha \partial_\tau t^j v) (x,t),
$$
if  $|\alpha|+2j\leq 2$. Therefore 
$$
\lim_{\delta \to + 0}\| 
\partial_x^\alpha \partial_t^j v  - 
\partial_x^\alpha \partial_t^j (R _\delta v) 
\|_{L^2 ({\mathbb R}^{n+1})} =0 \mbox{ if } |\alpha|+2j\leq 2,
$$
and then 
$$
\lim_{\delta \to + 0}\| v- R _\delta v  
\|_{H^{2,1} ({\mathbb R}^{n+1})} =0 .
$$
Next, we may continue the proof with the use standard scheme, see, for example, 
Indeed, denote by  $\partial_ \nu=\sum\limits_{j=1}^n\nu_j \partial_{x_j}$ the normal 
derivative, where  $\nu (x)=(\nu_1 (x), ..., \nu_n (x))$ is the unit external normal vector 
to the surface $\partial\Omega$ at the point  $x$. If 
 $\partial\omega$ is a surface of class $C^2$, then, as the function 
$v$ belongs to space \eqref{eq.space.1}, we see that there are the traces 
$$
v_{|\partial (\omega_{T_1,T_2})} \in H^{1/2} (\partial (\omega_{T_1,T_2})), 
$$
$$
 v_{|(\partial \omega)_{T_1,T_2}} \in L^2 ([T_1, T_2], H^{3/2} (\partial \omega))
, \, \, 
\partial_\nu  v_{|(\partial \omega)_{T_1,T_2}} 
\in  L^2 ([T_1, T_2], H^{1/2} (\partial \omega)),
$$
cf. \cite[ch. 3, \S 7, property 7]{MikhX}. 

Besides, according to \eqref{eq.v2}, 
$$
v = 0  \mbox{ on } \partial (\omega_{T_1,T_2}), 
\, \, 
\partial_ \nu v = 0  \mbox{ on } (\partial \omega)_{T_1,T_2}.
$$
Hence, the spectral synthesis theorem, see \cite{HedbWolf1}, implies that 
$v$ belongs to both the space  \eqref{eq.space.1} and the space  
\begin{equation} \label{eq.space.2}
 C ( [T_1,T_2], H^{1}_0(\omega) )
\cap L^2 ([T_1, T_2], H^{2}_0 (\omega) ) \cap 
H^{1}_0 (\omega_{T_1,T_2}) 
\cap H^{2,1} (\omega_{T_1,T_2})  ,
\end{equation}
where $ H^{s}_0(\omega) $ is the closure in $ H^{s}(\omega) $ of the space  
$C^\infty_0 (\omega)$ of infinitely differentiable functions with compact supports in 
$\omega$. 

On the other hand, if $\partial\omega \in C^2 $, then there is a real valued function 
$\rho $, two times continuously differentiable in a neghbourhood 
 $U$ of the surface $\partial\omega$ and such that 
$$
\omega = \{ x \in {\mathbb R}^n: \, \rho (x)<0\}, \,\, \nabla \rho \ne 0 \mbox{ in } U 
\} .
$$
Hence, for all sufficiently small numbers $\varepsilon>0$ the sets 
$$
\omega^{\varepsilon} = 
\{ x \in {\mathbb R}^n: \, \rho (x)<-\varepsilon\}
$$
are domains with boundaries of class $C^2$ and 
$$\omega ^{\varepsilon}\Subset \omega ^{\varepsilon'}\Subset 
\omega, 
$$
if $0<\varepsilon' <\varepsilon$; moreover for the Lebesgue 
measure of the domain  $\omega \setminus 
\overline{\omega ^{\varepsilon} }$ we have  
$$
\lim_{\varepsilon \to +0}\mbox{mes} (\omega \setminus 
\overline{\omega ^{\varepsilon} }) = 0. 
$$
According to \cite[ch. 3, \S 5, lemma 1]{MikhX}, if  $\partial \omega \in C^1$ then 
there is a constant  $C _0 (\partial \omega)$, depending on the square of the surface 
$\partial \omega$, only, and such that  
$$
 \|\tilde v \|_{L^2 (\omega \setminus \overline {\omega^\varepsilon})} 
\leq C_0 (\partial \omega)\, \varepsilon  
\|\tilde v \|_{H^1 (\omega \setminus \overline {\omega^\varepsilon})} 
$$
for any function $\tilde v \in H^1 (\omega)$ with zero trace $\tilde v_{|\partial \omega}$
on $\partial \omega$. 

Similarly, if  $\partial \omega \in C^2$, then there are constants  
 $C _1 (\partial \omega)$, $C _2 (\partial \omega)$, depending on the square 
of the surface  $\partial \omega$, and such that 
\begin{equation} \label{eq.est.1}
 \|\tilde v \|_{L^2 (\omega \setminus \overline {\omega^\varepsilon})} 
= C_1 (\partial \omega)\, \varepsilon^2 
\|\tilde v \|_{H^2 (\omega \setminus \overline {\omega^\varepsilon})} ,
\end{equation}
\begin{equation} \label{eq.est.2}
 \|\nabla \tilde v \|_{L^2 (\omega \setminus \overline {\omega^\varepsilon})} 
= C_2 (\partial \omega)\, \varepsilon  
\|\tilde v \|_{H^2 (\omega \setminus \overline {\omega^\varepsilon})} 
\end{equation}
for any function $\tilde v \in H^2 (\omega)$ with zero traces 
$\tilde v_{|\partial \omega}$ and $\partial_\nu \tilde v_{|\partial \omega}$
on $\partial \omega$. 

Set 
$$
R^{(1)}_\varepsilon (x) = \int_{\omega_{\varepsilon/2}}  h_{\varepsilon/3} (x-y) \, dy
$$
$$
R^{(2)}_\varepsilon (t) = \int_{T_1+\varepsilon/2}^{T_2-\varepsilon/2} 
 h_{\varepsilon/3} (t-\tau) \, dy.
$$
It is known, see, for instance, \cite[ch. 3, \S 5]{MikhX}, that
\begin{equation} \label{eq.Bochner.1}
0\leq R^{(1)}_\varepsilon  \leq 1, \,\,   
0 \leq R^{(2)}_\varepsilon  \leq 1,  
\end{equation}
\begin{equation} \label{eq.Bochner.2}
|\partial^\alpha R^{(1)}_\varepsilon | \leq c_{\alpha} \, \varepsilon ^{-|\alpha|}, 
\,\, 
|\partial^j  R^{(2)}_\varepsilon | \leq c_{j} \, \varepsilon ^{-j}, 
\end{equation}
for all $x \in {\mathbb R}^n$, $t \in \mathbb R$, 
$\alpha \in {\mathbb Z}^n_+$, $j \in {\mathbb Z}_+$, 
with some positive constants $c_\alpha, c_j$, independent on  $x$ and $t$.

Fix a sequence $\{ v_k\} \subset C^\infty _0({\mathbb R}^{n+1})$, converging to 
$v$ in the space $H^{2,1} ({\mathbb R}^{n+1} )$. Then the functional sequence 
$$
 \{  v_{k,\varepsilon} (x,t)=  R^{(2)}_\varepsilon (t) R^{(2)}_\varepsilon (x)v_k (x,t)\}
$$
lies to $C^\infty _0 (\omega _{T_1, T_2})$. 

By the triangle inequality, 
\begin{equation} \label{eq.vk.1}
\|v - v_{k,\varepsilon} \|_{H^{2,1} (\omega _{T_1, T_2})} \leq 
\|v - v_{k} \|_{H^{2,1} (\omega _{T_1, T_2})} + 
\|v _k - v_{k,\varepsilon} \|_{H^{2,1} (\omega _{T_1, T_2})} .
\end{equation}
As 
\begin{equation} \label{eq.lim.1}
\lim_{k\to + \infty}\|v - v_{k} \|_{H^{2,1} (\omega _{T_1, T_2})} =0,
\end{equation}
then we need to estimate the second summand in the right hand side of formula 
 \eqref{eq.vk.1}, only. However,  
$$
v _k (x,t)- v_{k,\varepsilon} (x,t) = 
(1-R^{(1)}_\varepsilon (t) R^{(2)}_\varepsilon (x)) v_k (x,t)
$$
and, in particular,  
\begin{equation} \label{eq.eps.0}
v _k (x,t)- v_{k,\varepsilon} (x,t) = 0 \mbox{ for all } (x,t) \in 
\omega^{\varepsilon} \times (T_1+\varepsilon, T_2-\varepsilon)
\end{equation}
Hence, 
\begin{equation} \label{eq.eps}
2^{-1}\|v _k - v_{k,\varepsilon}\|^2_{H^{2,1} (\omega_{T_1,T_2})} \leq  
\sum_{|\alpha|+2j \leq 2}
\|(1-R^{(1)}_\varepsilon  R^{(2)}_\varepsilon) \partial^\alpha _x \partial^j_t v_k \|^2
_{L^{2} (\omega_{T_1,T_2})} + 
\end{equation}
$$
\|R^{(1)}_\varepsilon  \Big(\frac{d R^{(2)}_\varepsilon}{dt} \Big) v_k \|^2
_{L^{2} (\omega _{T_1,T_1+\varepsilon} \cup \omega _{T_2-\varepsilon,T_2})} +
\sum_{1\leq |\alpha| \leq 2}
\|(\partial^{\alpha}R^{(1)}_\varepsilon)  R^{(2)}_\varepsilon  v_k \|^2
_{L^{2} ((\omega \setminus \omega^{\varepsilon})_{T_1,T_2})}  +
$$
$$
\sum_{|\beta| =1} \sum_{|\gamma|=1}
\|(\partial^{\beta}R^{(1)}_\varepsilon)  R^{(2)}_\varepsilon  \partial^{\gamma} v_k \|^2
_{L^{2} ((\omega \setminus \omega^{\varepsilon})_{T_1,T_2})}  .
$$
Since $v$ belongs to the space  \eqref{eq.space.2}, then 
\eqref{eq.est.1}, \eqref{eq.Bochner.1}, \eqref{eq.Bochner.2} 
and the Fubini theorem imply that  
\begin{equation} \label{eq.eps.1}
\sum_{1\leq |\alpha| \leq 2}
\|(\partial^{\alpha}R^{(1)}_\varepsilon)  R^{(2)}_\varepsilon  v \|^2
_{L^{2} ((\omega \setminus \omega^{\varepsilon})_{T_1,T_2})} \leq
\tilde C \sum_{1\leq |\alpha| \leq 2} \varepsilon^{2-|\alpha|} \int_{T_1}^{T_2} 
\|v \|^2_{H^{2} (\omega \setminus \omega^{\varepsilon})}\, dt\leq 
\end{equation}
$$
C \|  v \|^2
_{H^{2,1} ((\omega \setminus \omega^{\varepsilon})_{T_1,T_2})} 
\leq C \| v\|_{H^{2,1} (\omega_{T_1,T_2} \setminus 
\overline{\omega^{\varepsilon} _{T_1+\varepsilon, T_2-\varepsilon}})}
$$
and, similarly, 
\begin{equation} \label{eq.eps.2}
\sum_{|\beta| =1} \sum_{|\gamma|=1}
\|(\partial^{\beta}R^{(1)}_\varepsilon)  R^{(2)}_\varepsilon  \partial^{\gamma} v \|^2
_{L^{2} ((\omega \setminus \omega^{\varepsilon})_{T_1,T_2})}  \leq 
\end{equation}
$$
\tilde C \sum_{ |\beta |=1} \varepsilon^{1-|\beta|} \int_{T_1}^{T_2} 
\|v \|^2_{H^{2} (\omega \setminus \omega^{\varepsilon})} \, dt\leq 
C \|  v \|^2
_{H^{2,1} ((\omega \setminus \omega^{\varepsilon})_{T_1,T_2})} 
\leq C \| v\|_{H^{2,1} (\omega_{T_1,T_2} \setminus 
\overline{\omega^{\varepsilon} _{T_1+\varepsilon, T_2-\varepsilon}})}
$$
with constants $C,\tilde C$, independent on $v$ and $\varepsilon$. 

The boundaries of the  cylinder domains
$\omega _{T_1,T_1+\varepsilon}$ and $\omega _{T_2-\varepsilon,T_2}$ are not smooth, 
but combining results \cite[ch. 3, \S 5]{MikhX} 
related to a function $v$, having the trace vanishing on surfaces
 $\omega\times T_1$ and $\omega\times T_2$, 
with bounds  \eqref{eq.Bochner.1}, \eqref{eq.Bochner.2}, 
we see that 
\begin{equation} \label{eq.eps.3}
\|R^{(1)}_\varepsilon  \Big(\frac{d R^{(2)}_\varepsilon}{dt} \Big) v \|^2
_{L^{2} (\omega _{T_1,T_1+\varepsilon} \cup \omega _{T_2-\varepsilon,T_2})} 
\leq \varepsilon^{-1} \varepsilon 
C \, \| v\|^2
_{H^{1} (\omega _{T_1,T_1+\varepsilon} \cup \omega _{T_2-\varepsilon,T_2})} \leq 
 \end{equation}
$$
C \, \| v \|^2
_{H^{2,1} (\omega _{T_1,T_1+\varepsilon} \cup \omega _{T_2-\varepsilon,T_2})} 
\leq C \| v\|_{H^{2,1} (\omega_{T_1,T_2} \setminus 
\overline{\omega^{\varepsilon} _{T_1+\varepsilon, T_2-\varepsilon}})}
 $$
with a constant $C$, independent on $v$ and  $\varepsilon$.

Besides, according to  \eqref{eq.Bochner.1}, \eqref{eq.Bochner.2}, 
\begin{equation} \label{eq.eps.4}
\sum_{|\alpha|+2j \leq 2}
\|(1-R^{(1)}_\varepsilon  R^{(2)}_\varepsilon) \partial^\alpha _x \partial^j_t v \|^2
_{L^{2} (\omega_{T_1,T_2})}
\leq C \| v\|_{H^{2,1} (\omega_{T_1,T_2} \setminus 
\overline{\omega^{\varepsilon} _{T_1+\varepsilon, T_2-\varepsilon}})}
\end{equation}
with a constant $C$, independent on $v$ and  $\varepsilon$.

Using the continuity of the Lebesgue integral with respect 
to the measure of the integration set, we conclude that 
\begin{equation} \label{eq.lim.2}
\lim_{\varepsilon \to + 0}
\| v\|_{H^{2,1} (\omega_{T_1,T_2} \setminus 
\overline{\omega^{\varepsilon} _{T_1+\varepsilon, T_2-\varepsilon}})} =0.
\end{equation}

Fix a number $E>0$. Relation \eqref{eq.lim.1} means that there is a number  
 $N(E,\varepsilon) \in \mathbb N$ such that for all 
 $k\geq N(E,\varepsilon)$ we have 

$$
\|v - v_{k} \|_{H^{2,1} (\omega _{T_1, T_2})} < E \varepsilon ^2.
$$
In this case,  \eqref{eq.eps.4} implies that for such $k$ we have 
\begin{equation} \label{eq.eps.5}
2^{-1}\sum_{|\alpha|+2j \leq 2}
\|(1-R^{(1)}_\varepsilon  R^{(2)}_\varepsilon) \partial^\alpha _x \partial^j_t v_k \|^2
_{L^{2} (\omega_{T_1,T_2})} \leq 
\end{equation}
$$
 \sum_{|\alpha|+2j \leq 2}\Big(
\|(1-R^{(1)}_\varepsilon  R^{(2)}_\varepsilon) \partial^\alpha _x \partial^j_t (v_k-v) \|^2
_{L^{2} (\omega_{T_1,T_2})} + 
\|(1-R^{(1)}_\varepsilon  R^{(2)}_\varepsilon) \partial^\alpha _x \partial^j_t v \|^2
_{L^{2} (\omega_{T_1,T_2})}\Big)\leq
$$
$$
C (E\varepsilon^2 +
 \| v\|_{H^{2,1} (\omega_{T_1,T_2} \setminus 
\overline{\omega^{\varepsilon} _{T_1+\varepsilon, T_2-\varepsilon}})})
$$
with a constant $C$, independent on $v$ and  $\varepsilon$.

Besides,  \eqref{eq.eps.1} and \eqref{eq.eps.2} yield 
\begin{equation} \label{eq.eps.6}
2^{-1} \sum_{1\leq |\alpha| \leq 2}
\|(\partial^{\alpha}R^{(1)}_\varepsilon)  R^{(2)}_\varepsilon  v_k \|^2
_{L^{2} ((\omega \setminus \omega^{\varepsilon})_{T_1,T_2})}\leq 
\end{equation}
$$
\sum_{1\leq |\alpha| \leq 2}\Big(
\|(\partial^{\alpha}R^{(1)}_\varepsilon)  R^{(2)}_\varepsilon  (v_k-v) \|^2
_{L^{2} ((\omega \setminus \omega^{\varepsilon})_{T_1,T_2})} + 
\|(\partial^{\alpha}R^{(1)}_\varepsilon)  R^{(2)}_\varepsilon  v \|^2
_{L^{2} ((\omega \setminus \omega^{\varepsilon})_{T_1,T_2})}
\Big)\leq 
$$
$$
C (E+
 \| v\|_{H^{2,1} (\omega_{T_1,T_2} \setminus 
\overline{\omega^{\varepsilon} _{T_1+\varepsilon, T_2-\varepsilon}})}),
$$
and 
\begin{equation} \label{eq.eps.7}
2^{-1}\sum_{|\beta| =1} \sum_{|\gamma|=1}
\|(\partial^{\beta}R^{(1)}_\varepsilon)  R^{(2)}_\varepsilon  (\partial^{\gamma} v_k \|^2
_{L^{2} ((\omega \setminus \omega^{\varepsilon})_{T_1,T_2})}  \leq 
\end{equation}
$$
\sum_{|\beta| =1\atop |\gamma|=1} \Big(
\|(\partial^{\beta}R^{(1)}_\varepsilon)  R^{(2)}_\varepsilon  \partial^{\gamma} (v_k-v) \|^2
_{L^{2} ((\omega \setminus \omega^{\varepsilon})_{T_1,T_2})} + 
\|(\partial^{\beta}R^{(1)}_\varepsilon)  R^{(2)}_\varepsilon  (\partial^{\gamma} v \|^2
_{L^{2} ((\omega \setminus \omega^{\varepsilon})_{T_1,T_2})}\Big) \leq 
$$
$$
C (E\varepsilon +
 \| v\|_{H^{2,1} (\omega_{T_1,T_2} \setminus 
\overline{\omega^{\varepsilon} _{T_1+\varepsilon, T_2-\varepsilon}})})
$$
with a constant $C$, independent on $v$ and  $\varepsilon$.

Similarly, using \eqref{eq.eps.3}, we obtain 
\begin{equation} \label{eq.eps.8}
\|R^{(1)}_\varepsilon  \Big(\frac{d R^{(2)}_\varepsilon}{dt} \Big) v_k \|^2
_{L^{2} (\omega _{T_1,T_1+\varepsilon} \cup \omega _{T_2-\varepsilon,T_2})}
\leq 
\end{equation}
$$
\|R^{(1)}_\varepsilon  \Big(\frac{d R^{(2)}_\varepsilon}{dt} \Big) (v_k -v)\|^2
_{L^{2} (\omega _{T_1,T_1+\varepsilon} \cup \omega _{T_2-\varepsilon,T_2})} + 
\|R^{(1)}_\varepsilon  \Big(\frac{d R^{(2)}_\varepsilon}{dt} \Big) v\|^2
_{L^{2} (\omega _{T_1,T_1+\varepsilon} \cup \omega _{T_2-\varepsilon,T_2})}
\leq 
$$ 
$$
C (E\varepsilon +
 \| v\|_{H^{2,1} (\omega_{T_1,T_2} \setminus 
\overline{\omega^{\varepsilon} _{T_1+\varepsilon, T_2-\varepsilon}})})
$$
with a constant $C$, independent on $v$ and  $\varepsilon$.

Finally, combining estimates \eqref{eq.eps}, 
\eqref{eq.eps.5}--\eqref{eq.eps.8} and taking in account 
\eqref{eq.lim.2}, we conclude that the statement of the lemma is fulfilled.
\end{proof}

Now, using lemma \ref{l.dense.c0} and fixing a sequence 
 $\{ v_k \} \subset C^\infty _0 
(\omega _{T_1, T_2})$, converging to the function $v$ in $H^{2,1} (\omega _{T_1, T_2})$  
we see that 
$$
\|u\|^2_{L^2 (\omega _{T_1, T_2})} = 
(u , {\mathcal H}' v )_{L^2 (\omega _{T_1, T_2})} = 
\lim_{k \to + \infty }(u, {\mathcal H}' v_ k )_{L^2 (\omega _{T_1, T_2})}   
= 0,
$$
because ${\mathcal H} u =0 $ in $\omega _{T_1, T_2}$ in the sense of distributions. 

Thus, $u \equiv 0$ in $\omega _{T_1, T_2}$, that was to be proved.

\textit{Necessity.} This part of the proof inspired by the arguments 
from the  classical approximation theorem for spaces of solutions to the heat equation in 
domains from  ${\mathbb R}^{n+1}$ with the topology of the uniform convergence on 
subcompacts, see \cite{J}. More precisely, let the complement 
$\Omega \setminus  \omega$ has at least one compact component. 
As we noted before, since the domains  $\Omega, \omega$ has smooth boundaries, 
this component is the closure of a non-empty domain 
 $\omega^{(0)}$. Moreover, the set 
$\omega \cup \overline{\omega^{(0)}}$ is a domain with smooth boundary in ${\mathbb R}^n$. 

Fix a point  $(x_0,t_0)\in \omega^{(0)} \times (T_1,T_2)$. According to 
\cite[lemma from \S 1]{J}, for any $\delta>0$ there is a function 
 $v_0 \in S _{\mathcal H}({\mathbb R}^{n+1})$ such that 
$v(x_0, t_0) \ne 0$ and $v_0(x,t) =0$ for all $t$, $|t-t_0|\geq \delta$. 
 
Next, there is an infinitely times differentiable function 
 $\phi$ supported in $\omega $ such that $\phi (x_0) \equiv 1$ in 
a neighbourhood  $U $ of the compact 
 $\overline \omega^{(0)} $.  Then the function $v_1 (x,t)= \phi(x) v_0(x,t)$
is infinitely smooth in ${\mathbb R}^{n+1}$, supported in 
$\omega \times [t_0-\delta, t_0+\delta]$ and, moreover, 
$$
  {\mathcal H'} v_1 = 2\nabla \phi \cdot \nabla v_0 + v_0 \Delta \phi 
	\mbox{ in }   {\mathbb R}^{n+1}.
$$
In particular, since $ \nabla \phi =0$ in $U$, then 
$$
  {\mathcal H'} v_1 =0 \mbox{ in } U \times (T_1,T_2).
$$
Denote by $\Pi_0$ the orthogonal projection from $L^2 (\omega_{T_1,T_2})$ 
onto $L^2_{\mathcal H} (\omega_{T_1,T_2})$. 

The properties of the projection $\Pi_0$, the function $v_1$ and 
the fundamental solution  $\Phi$ imply that  for all  
$ (y,\tau) \not \in \overline{ \omega\times (T_1,T_2)}$ we have 
$$
\int_{T_1}^{T_2} \int_{\omega} 
(\Pi_0 {\mathcal H'} v_1)  (x,y) \Phi (x-y, t-\tau) dx \, dt =
$$
$$
\int_{T_1}^{T_2} \int_{\omega} 
({\mathcal H'} v_1)  (x,y) \Phi (x-y, t-\tau) dx \, dt =
$$
\begin{equation} \label{eq.non.ort}
\int_{T_1}^{T_2} \int_{\omega\cup \omega^{(0)}} 
({\mathcal H'} v_1)  (x,y) \Phi (x-y, t-\tau) dx \, dt =v_1 (y,\tau).
\end{equation}
As a corollary, the function  $\Pi_0 {\mathcal H'} v_1 \in 
L^2 _{\mathcal H}(\omega_{T_1,T_2})$ is not 
$L^2 (\omega_{T_1,T_2})$-orthogonal to the function 
$\Phi (x-x_0, t- t_0) \in L^2 _{\mathcal H}(\omega_{T_1,T_2})$, but it is 
$L^2 (\omega_{T_1,T_2})$-orthogonal to the functions 
 $\Phi (x- y, t- \tau) \in L^2 _{\mathcal H}(\omega_{T_1,T_2})$ 
with any vectors $(y,\tau) \not \in \overline {(\omega \cup \overline{\omega^{(0)}}) 
\times (T_1,T_2)}$.

To finish the proof we need the integral Green formula for the heat equation. 
With this purpose, for functions $f \in  L^2(\Omega_{T_1,T_2})$, $v \in L^2 ([T_1,T_2], 
H^{1/2}(\partial \Omega))$, $w \in L^2 ([T_1,T_2], 
H^{3/2}(\partial \Omega))$, $h \in 
H^{1/2}(\Omega)$ we consider the following parabolic potentials:  
\begin{equation*} 
I_{\Omega, T_1} (h) (x,t)= \int\limits_{\Omega}\Phi(x-y, t)h(y,T_1) dy,  
$$
$$
G_{\Omega,T_1} (f) (x,t)=\int\limits_{T_1}^t\int\limits_\Omega \Phi(x-y,\ t-\tau)f(y, \tau)
dy d\tau, 
\end{equation*}
\begin{equation*} 
V_{\partial \Omega,T_1} (v) (x,t)=\int\limits_{T_1}^t\int\limits_{\partial \Omega} \Phi(x-y, t-\tau) v(y, \tau)
ds(y)d\tau, 
\end{equation*}
\begin{equation*} 
W_{\partial \Omega,T_1} (w) (x,t)= - \int\limits_{T_1}^t\int\limits_{\partial \Omega} \partial_{\nu_y} \Phi(x-y, 
t-\tau)w(y, \tau)ds(y) d\tau
\end{equation*}
(see, for instance, 
\cite[ch. 1, \S 3 and ch. 5, \S 2]{frid}. 
By the definition, these are  (improper) integrals, depending 
on the parameter $(x,t)$.

\begin{lem} \label{l.Green}
For any $ T_1 < T_2$ and any  $u \in H^{2,1} (\Omega_{T_1,T_2}) $, the following 
formula holds true:
\begin{equation*} 
\left.
\begin{aligned}
u(x, t) \mbox{ in } \Omega_{T_1,T_2}  \\
0 \mbox{ outside } \overline{\Omega_{T_1,T_2}} 
\end{aligned}
\right\} \! = 
I_{\Omega,T_1} (u)   + G_{\Omega,T_1} ({\mathcal H}u)  + 
V _{\partial \Omega, T_1} \left( 
\partial _\nu u \right)  +  W_{\partial \Omega, T_1} (u) .
\end{equation*}
\end{lem}

\begin{proof}  See, \cite[ch. 6, \S 12]{svesh} 
(and \cite[theorem 2.4.8]{Tark37} for more general differential  operators, 
having fundamental solutions or parametrices).
\end{proof}

If a function $u $ belongs to  $S _{\mathcal H}(\overline{ \Omega_{T_1,T_2}}) $, 
then it belongs to  $H^{2,1} _{\mathcal H}(\Omega'_{T'_1,T'_2}) $ 
for some numbers  $T_1'<T_1<T_2<T_2'$ and a bounded domain 
$\Omega'\Supset \Omega$. Then Green formula yields
\begin{equation*} 
\left.
\begin{aligned}
u(x, t) \mbox{ in } \Omega'_{T'_1,T'_2}  
\\
0 \mbox{ outside } \overline{\Omega'_{T'_1,T'_2}} 
\end{aligned}
\right\} \! = 
I_{\Omega',T'_1} (u)    + 
V _{\partial \Omega, T'_1} \left( 
\partial _\nu u \right)  +  W_{\partial \Omega', T'_1} (u). 
\end{equation*} 

In particular,  Fubini theorem and formulas  \eqref{eq.non.ort} for 
$(y,\tau) \in (\partial \Omega' \times [T_1',T'_2]) \cap (\Omega' \times \{0\})$ give us 
possibility to conclude that the non-zero function 
$\Pi_0 {\mathcal H'} v_1 \in L^2 _{\mathcal H}(\omega_{T_1,T_2})$ is 
$L^2 (\omega_{T_1,T_2})$-orthogonal to all the functions from  
$S _{\mathcal H}(\overline{ \Omega_{T_1,T_2}})$. This proves that 
 $S _{\mathcal H}(\overline {\Omega_{T_1,T_2}})$ is not everywhere dense set in the space 
$L^2 _{\mathcal H}(\omega_{T_1,T_2})$ if there is a compact components of the set 
$\Omega \setminus \omega$ in $\Omega$.
\end{proof}

\begin{cor} \label{c.dense.base0}
Let $s,k\in {\mathbb Z}_+$ be arbitrary numbers, 
$\omega \subset \Omega \Subset {\mathbb R}^n$, $\partial \omega, \partial \Omega 
\in C^2$ and let the complement
$\Omega \setminus \omega$ has no compact components in 
$\Omega$. Then the spaces 
$C^\infty_{\mathcal H}(\overline {\Omega_{T_1,T_2}})$
and $H^{k,2s,s} _{\mathcal H}(\Omega_{T_1,T_2})$ everywhere dense 
in  $L^{2} _{\mathcal H}(\omega_{T_1,T_2})$. 
\end{cor}

\begin{proof} Follows immediately from theorem \ref{t.dense.base}, 
because of embeddings \eqref{eq.emb.H}.
\end{proof}

As we noted in the introduction, assumptions of 
theorem \ref{t.dense.base} are quite similar to that 
of the Runge type theorems related to the uniform approximation on compact subsets 
for solutions to the heat equation in a lesser domain by the solutions in a bigger one, 
see \cite{J}, \cite{D80} 
(and a refinement  \cite{GauTa10} related with a constructive way of approximation 
sequences with the use of the fundamental solution to the  heat equation). 
 It is appropriate to note, that instead of the cylinder domains of the type 
 $\omega_{T_1,T_2}$ one may consider more general domains with additional assumptions 
on the boundaries' smoothness. 

\section{Theorem on the basis with the double orthogonality property}
\label{s.ex}
\setcounter{equation}{0}

We continue with the theorem on the basis with the double orthogonality property
in spaces of solutions to the heat equation.

\begin{thm} \label{t.bdo} 
Let $s\in \mathbb N$, $k \in {\mathbb Z}_+$, and let 
$\omega$ be a subdomain in  $\Omega \Subset {\mathbb R}^n$ with $C^2$-boundary and such that 
the complement  $\Omega\setminus \omega$ has no compact components in 
 $\Omega$. Then there is an orthonormal basis 
  $\{ b_\nu\}$ in the space $H^{k,2s,s}_{\mathcal H} 
(\Omega_{T_1,T_2})$ such that its restriction  
$\{ b_{\nu|\omega_{T_1,T_2}}\}$ to $\omega_{T_1,T_2}$ is 
an orthogonal basis in $L^{2}_{\mathcal H} (\omega_{T_1,T_2})$. 
\end{thm} 

\begin{proof} By the definition, for numbers $s\in \mathbb N$ and  $k \in {\mathbb Z}_+$,  the 
space $H^{k,2s,s} _{\mathcal H} (\Omega_{T_1,T_2})$  is embedded continuously to 
the space  $L^2 _{\mathcal H} (\omega_{T_1,T_2})$. Denote by  $R_{\Omega,\omega}$ the natural 
embedding operator  
$$
R_{\Omega,\omega}: H^{k,2s,s} _{\mathcal H}(\Omega_{T_1,T_2})\to 
L^2 _{\mathcal H}(\omega_{T_1,T_2}).
$$ 
The analyticity of the solutions to the heat equation with respect to 
the space variables implies that the operator 
$R_{\Omega,\omega}$ is injective. Moreover, according to theorem 
\ref{t.dense.base}, the range of the operator $R_{\Omega,\omega}$ 
is everywhere dense  in the space $L^2 _{\mathcal H}(\omega_{T_1,T_2})$. 

By Fubini theorem, anisotropic Sobolev space 
$H^{2,1} _{\mathcal H}(\Omega_{T_1,T_2})$ is embedded continuously into the Bochner 
space ${\mathcal B}((T_1,T_2, H^2(\Omega),L^2 (\Omega))$, consisting of mappings 
$v: [T_1,T_2] \to H^2  (\Omega)$ such that  $\partial_t v\in L^2 (\Omega)$, see
 \cite[ch. 1, \S 5]{Lion69}. Rellich-Kondrashov theorem implies that the embedding   
$H^2 (\Omega) \to L^2 (\Omega)$ is compact. Using the well known theorem on the 
compact embedding for the Bochner type spaces  
(see, for instance, \cite[ch. 1, \S 5, theorem 5.1]{Lion69}), we conclude that 
the space ${\mathcal B}((T_1,T_2, H^2(\Omega),L^2 (\Omega))$ is embedded compactly 
into $L^2 ((T_1,T_2),\Omega) = L^2 (\Omega_{T_1, T_2})$. Thus, the space  
$H^{2,1} _{\mathcal H}(\Omega_{T_1, T_2})$  is embedded compactly into  
$L^2_{\mathcal H} (\Omega_{T_1, T_2})$ and, of course, into 
 $L^2 _{\mathcal H}(\omega_{T_1,T_2})$. Therefore, the space 
 $H^{k,2s,s} _{\mathcal H}(\Omega_{T_1, T_2})$  is embedded compactly into  
$L^2_{\mathcal H} (\omega_{T_1, T_2})$, too, i.e. 
the operator $R_{\Omega,\omega}$ is compact. 

Denote by $R_{\Omega,\omega}^*$ the adjoint mapping for the operator 
$R_{\Omega,\omega}$  in the theory of the Hilbert spaces, i.e. 
$R_{\Omega,\omega}^*: L^2 _{\mathcal H}(\omega_{T_1,T_2}) \to 
H^{k,2s,s} _{\mathcal H}(\Omega_{T_1,T_2}) $. By the Hilbert-Schmidt Theorem, 
there is an orthonormal basis $\{b_\nu\}$ in the space 
$H^{k,2s,s} _{\mathcal H}(\Omega_{T_1,T_2})$, consisting of the eigen-vectors 
of the compact self-adjoint operator  
 $R_{\Omega,\omega}^* R_{\Omega,\omega}: 
H^{k,2s,s} _{\mathcal H}(\Omega_{T_1,T_2}) \to H^{k,2s,s} _{\mathcal H}(\Omega_{T_1,T_2})$.  
Finally, using results of 
\cite[Example 1.9]{ShTaLMS}, we conclude that the system of vectors 
 $\{b_\nu\}$ is the basis with the double orthogonality property looked for. 
\end{proof}

\begin{rem} \label{r.integral} 
It was shown in \cite[theorem 6.5]{ShTaLMS} that the operator 
$R_{\Omega,\omega}^*R_{\Omega,\omega}$
may be identified as an integral one. Indeed, 
by the Sobolev embedding theorem follow that for 
sufficiently large $s$ and  $k$ the space $ H^{k,2s,s}_{\mathcal H} (\Omega_{T_1,T_2})$ 
is embedded continuously into the normed space of continuous functions  
$C(\overline{\Omega_{T_1,T_2}})$  on the compact $\overline{\Omega_{T_1,T_2}}$ from 
${\mathbb R}^{n+1}$. Thus, it is a Hilbert space with the reproducing kernel  
(see \cite{Ar50}). Besides, as the heat operator is hypoelliptic ,
the elements of the space $ H^{k,2s,s}_{\mathcal H} (\Omega_{T_1,T_2})$ 
are smooth on $\Omega_{T_1,T_2}$. That is why there is a kernel 
 ${\mathcal K}(x,t,y,\tau) \in C_{loc} ^\infty (\Omega_{T_1,T_2}
\times \Omega_{T_1,T_2})$ such that 
$$
u(x) = (u,{\mathcal K}(x,t,\cdot, \cdot))_{H^{k,2s,s}_{\mathcal H} (\Omega_{T_1,T_2})} \mbox{ 
for all } u \in  H^{k,2s,s}_{\mathcal H} (\Omega_{T_1,T_2}), \, (x,t) \in \Omega_{T_1,T_2}.
$$
If  $\{ e_\nu \}_{\nu =1}^{\infty}$ is an orthonormal basis in the Hilbert space 
$H^{k,2s,s}_{\mathcal H} (\Omega_{T_1,T_2})$ then for all 
$(x,t) \in \Omega_{T_1,T_2}$  we have
$$
{\mathcal K } (x,t,y,\tau)= \sum_{j =1}^{\infty}  
e_j (x,t) e_j  (y,\tau),
$$ 
where the series converges in $H^{k,2s,s}_{\mathcal H} (\Omega_{T_1,T_2})$ with respect 
to variables $(x,t)$ for each pair  $(y,\tau) \in \Omega_{T_1,T_2}$. 
As a series of variables $(x,t,y,\tau) \in \Omega_{T_1,T_2}\times \Omega_{T_1,T_2}$, it 
converges uniformly on compacts from  $\Omega_{T_1,T_2}\times \Omega_{T_1,T_2}$.

Now, simple calculations show that 
$$
(R^* R u)(x,t)  = (u,{\mathcal K } (x,t,\cdot,\cdot))_{H^{k,2s,s} _{\mathcal H} 
(\Omega_{T_1,T_2})}, \, (x,t )\in \Omega_{T_1,T_2}.
$$
\end{rem}

However, it is not easy to construct an example of a basis with the double 
orthogonality property, provided by theorem  \ref{t.bdo}. A non-complete double 
orthogonal countable (trigonometric) system was constructed in \cite{KuSh} 
for cubes  $\omega$ and $\Omega$ in ${\mathbb R}^n$ if 
their centres coincide and the ratio of their edges equals to two. 
Let us indicate one more example; it is related to the case 
where the cylinder bases of   
$\omega_{T_1,T_2}$, $\Omega_{T_1,T_2}$ are concentric balls in ${\mathbb R}^n$. 

\begin{exmp}\label{ex.bdo}
Let  $0<R_1<R_2<+\infty$, and $\omega=B(0,R_1)$, $\Omega=B(0,R_2)$, 
where $B(0,R)$ is the ball of the radius $R$ in ${\mathbb R}^n$. 
In order to construct a system with the double orthogonality property 
for the cylinders $\omega_{T_1,T_2}$ and  
$\Omega_{T_1,T_2}$ we use eigenfunctions of the Laplace operator 
related to the Dirichet and the Neumann problems in 
$B(0,R_2)$. More precisely, after passing to the spherical coordinates 
 $x=r\,S(\varphi)$ with $\varphi$ being coordinates 
on the unit sphere  $\mathbb S$ in ${\mathbb R}^n$ one obtains
$$
\Delta
=\frac{1}{r^2}
\Big(\Big(r\frac{\partial}{\partial r}\Big)^2
+(n-2)\Big(r\frac{\partial}{\partial r}\Big)
-\Delta_{\mathbb S}
\Big),
$$
where $\Delta_{\mathbb  S}$ is the Laplace--Beltrami operator 
on the unit sphere  in ${\mathbb R}^n$. In order to solve 
the homogeneous equation
$$
(-\Delta+\lambda) u=0 \mbox{ in } B(0,R), 
$$ 
one usually uses Fourier method: look for $u$ in the form 
$u(r,\varphi)=g(r)h(\varphi)$, where  $g$ and $h$  
satisfy 
\begin{equation*}
\left\{
\begin{array}{lll}
\Big(\Big(r\frac{\partial}{\partial r}\Big)^2
+(n-2)\Big(r\frac{\partial}{\partial r}\Big)
-\lambda r^2
\Big)g
&=
a\,g
\\
\Delta_{\mathbb S}h
&=
a\,h,
\end{array}
\right.
\end{equation*}
with $a$ being an eigenvalue of the Laplace-Beltrami operator $\Delta_{\mathbb S}$.
It is well known that these eigenvalues equal to  $a=k(n+k-2)$, $k \in {\mathbb Z}_+ $, and 
the corresponding eigenfunctions are the spherical harmonics  
(see, for instance, \cite[ch. 4, \S 3]{TikhSamaX}). Then the first equation takes the form 
\begin{equation} \label{eq.7.4}
\Big(\Big(r\frac{\partial}{\partial r}\Big)^2
+(n-2)\Big(r\frac\partial{\partial r}\Big)
-\left(k(n+k-2)+\lambda r^2\right)
\Big)g
=0,
\end{equation}
and its solutions $g = g_k(r,\lambda)$ can be expresses via the Bessel functions 
$J_p$, $Y_p$:
$$
g_k (r,\lambda) = r^{(2-n)/2} \Big(C_1 J_{p_k} (\sqrt{\lambda} r)+ C_2 Y_{p_k} 
(\sqrt{\lambda} r)
\Big), 
$$
$$
p_k=\sqrt{(n-2)^2/4+ k^2(n+k-2)^2},
$$
with arbitrary constants  $C_1, C_2$, see \cite{Bow}, \cite[appendix 2]{TikhSamaX}).
For instance, for $\lambda=0$ we obtain $g_k(r,0)=C_1 r^k+C_2 r^{2-k-n}$; in this case  
$r^k h_k(\varphi)$ are spherical harmonics
(restrictions to the unit sphere of harmonic homogeneous polynomials 
$h_k$ of the degree $k$), forming a linear space of dimension   
$J(k)=\frac{(n+2k-2)(n+k-3)!}{k!(n-2)!}$.  Choosing an $L^2 (\partial B(0,1))$-orthonormal 
basis $\{ h^{(j)}_k(\varphi)\}$, $1\leq j \leq J(k)$, in the space of spherical functions 
we obtain a typical basis 
$\{ r^k h^{(j)}_k(\varphi)\}$ with the double orthogonality property in the 
spaces of Sobolev  harmonic functions in any ball with the centre  at the origin, see  
\cite{Sh96}. Moreover, the functions 
$$
B^{1,j}_{k,m} (r,\varphi) = r^{(2-n)/2}  J_{p_k} (\sqrt{\lambda_{k,m}^{(1)}} r) 
h^{(j)}_k(\varphi)
$$
form a system of eigenfunctions related to the Dirichlet problem for 
the Laplace operator in the ball
 $B(0,R_2)$ with the corresponding eigenvalues  
$\lambda_{k,m}^{(1)}$, where $\sqrt{\lambda_{k,m}^{(1)}}/R_2$ is the $m$-th zero 
of the Bessel function  $J_{p_k} $.  Similarly, the functions 
$$
B^{2,j}_{k,m} (r,\varphi) = r^{(2-n)/2}  J_{p_k} (\sqrt{\lambda_{k,m}^{(2)}} r) 
h^{(j)}_k(\varphi)
$$
form a system of eigenfunctions related to the Neumann problem for 
the Laplace operator in the ball $B(0,R_2)$, with the corresponding eigenvalues,
$\lambda_{k,m}^{(2)}$, where $\sqrt{\lambda_{k,m}^{(2)}}/R_2$ is the $m$-th zero 
of the derivative $J_{p_k}'$ of the Bessel function $J_{p_k} $,  see \cite[ch. 4, \S 3 and 
appendix 2]{TikhSamaX}). In the Cartesian coordinates we obviously have 
$$
B^{2,j}_{k,m} (x)  = |x|^{(2-n)/2-k}  J_{p_k} (\sqrt{\lambda_{k,m}^{(2)}} |x|) 
h^{(j)}_k( x).
$$
Clearly, the functions 
$$
E^{i,j}_{k,m} (x,t) 
=e^{\lambda_{k,m}^{(i)} t} B^{i,j}_{k,m} (x) 
\mbox{ and  } h^{(j)}_k (x) 
$$
satisfy the heat equation in ${\mathbb R}^{n+1}$.  Then the double orthogonality property 
means precisely that 
for $(k,j)\ne (k',j')$ and any numbers  $m',m\in \mathbb N$ the functions
$E^{i,j}_{k,m}$ and $E^{i',j'}_{k',m'} $ 
are orthogonal in the space  $H^{l,2s,s} (B(0,R) \times (T_1,T_2))$ with any 
$0<R<R_2$ and $l,s\in {\mathbb Z}_+$. This gives the possibility to choose 
from the set $\{E^{i,j}_{k,m} (x,t) ,  h^{(j)}_k(x)\}$ some finite or countable 
subsystems with non-repetitive pairs  $(k,j)$, 
$k \in {\mathbb Z}_+$, $1\leq j \leq J(k)$; obviously, 
these subsystems  have the double orthogonality property  in the spaces
$H^{l,2s,s}_{\mathcal H} (B(0,R_2) \times (T_1,T_2))$  and 
$H^{l',2s',s'}_{\mathcal H} (B(0,R_1) \times (T_1,T_2))$. Unfortunately, 
such subsystems are not complete. For example, the polynomial 
of the form  
$$
t \Delta G_{k+2} + G_{k+2} ,
$$
where $G_k$ is a homogeneous biharmonic polynomial of the degree  $k \in {\mathbb Z}_+$,  
satisfies the heat equation in ${\mathbb R}^{n+1}$. However, if any of the discussed 
above subsystems are complete $H^{2,1}_{\mathcal H} (B(0,R_2) \times (T_1,T_2))$ then 
it is possible to approximate the homogeneous harmonic polynomial  
$$
\Delta G_{k_0+2} = \Delta (t \Delta G_{k_0+2} + G_{k_0+2})
$$ 
of the degree $k_0$ in the space  $L^{2} (B(0,R_2) \times (T_1,T_2))$ 
by linear combinations of functions of the form 
$$
\Delta E^{i,j}_{k,m} (x,t) = \lambda_{k,m}^{(i)} E^{i,j}_{k,m} (x,t) ,\, 
\Delta h^{(j)}_ k \equiv 0. 
$$
This is impossible because the pairs $(k,j)$ are non-repetitive, and hence 
the linear combinations should be finite and containing the functions 
  $ E^{i_j,j}_{k_0,m_j} (x,t)$, $1\leq j \leq J(k_0)$, $m_j \in \mathbb N$, only, that 
	corresponds to the homogeneous harmonic polynomials 
$h^{(j)}_ {k_0}$ of the degree $k_0$. On the other hand, if we allow 
the pairs $(k,j)$ to enter these subsystems for different 
 $i$  and $m$, then the double orthogonality property will fail and 
we will need an additional orthogonalisation.  
\end{exmp}

\bigskip

{\bf Funding.} The research is supported by Sirius University of Science and Technology 
(project 'Spectral and Functional Inequalities of Mathematical Physics and Their Applications')

\end{document}